\documentclass[12pt,a4paper]{amsart}
\usepackage{amsfonts}
\usepackage{color}
\usepackage{amssymb}
\usepackage{graphicx}

\def\<{\langle}
\def\>{\rangle}
\def\-{\overline}

\newtheorem{theorem}{Theorem}[section]
\newtheorem{lemma}[theorem]{Lemma}

\newtheorem{remark}[theorem]{Remark}

\input{tcilatex}

\begin{document}
\def\jump{\vskip 0.3cm}

\author{Peter Scott}
\address{Mathematics Department\\
University of Michigan\\
Ann Arbor, Michigan 48109, USA.\\
email: pscott@umich.edu}
\author{Hamish Short}
\address{L.A.T.P. UMR 7353, Universit\'e d'Aix--Marseille, 39 rue Joliot
Curie, Marseille 13453, France.\\
email: hamish.short@latp.univ-mrs.fr}
\title[Homeo. problem for $3$--manifolds]{The homeomorphism problem for
closed $3$--manifolds}
\maketitle

\begin{abstract}
We give a more geometric approach to an algorithm for deciding whether two
hyperbolic $3$--manifolds are homeomorphic. We also give a more algebraic
approach to the homeomorphism problem for geometric, but non-hyperbolic, $3$%
--manifolds.
\end{abstract}

%\centerline{6th November}

%\hoffset -2cm

%\textwidth 12cm\textheigh 

\section{Introduction}

The homeomorphism problem for closed orientable triangulated $3$--manifolds
has been studied for many years, with partial results by many authors. The
work of Perelman \cite{Perelman1, Perelman2, Perelman3} proving Thurston's
Geometrization Conjecture finally allowed a complete solution for
irreducible such manifolds, which has been described by Jaco in his Beijing
Lectures \cite{J}, and also in the first chapter of the book by Bessi\`{e}%
res et al \cite{BBMBP}. The work of many previous authors is put together,
and different algorithms are used to deal with the Haken case, to find the
JSJ decomposition, and to deal with different geometries. In the case of two
closed hyperbolic $3$--manifolds $M_{1}$ and $M_{2}$, Mostow rigidity tells
us that $M_{1}$ and $M_{2}$ are homeomorphic if and only if their
fundamental groups are isomorphic. Thus the homeomorphism problem for $M_{1}$
and $M_{2}$ can be solved by appealing to Sela's solution \cite{S} of the
isomorphism problem for torsion-free word hyperbolic groups. The initial aim
of this paper was to give a more geometric approach to the homeomorphism
problem in this case, which avoids quoting Sela's work. But in addition, we
also give a more algebraic approach to some other parts of the homeomorphism
problem, though the geometric results of Jaco and Oertel \cite{JO} are still
needed for the existence of incompressible surfaces. We are not claiming
that the algorithms we present are in any way superior to those referred to
by Jaco and by Bessi\`{e}res et al. Our aim is simply to increase the range
of applicable algorithms from which to choose. In this paper, we will mostly
consider closed, orientable and irreducible $3$--manifolds. It is known that
given two triangulated closed orientable $3$--manifolds $M_{1}$ and $M_{2}$,
there is an algorithm to find the geometric structures on the geometric
pieces of each, and then there is an algorithm to decide whether or not the
pieces of $M_{1}$ are homeomorphic to the pieces of $M_{2}$ (see e.g. \cite[%
1.4]{M}). This algorithm is described in \cite[1.4.1]{BBMBP}, where it is
claimed that this solves the homeomorphism problem for triangulated closed
orientable $3$--manifolds. (A similar claim is made by Matveev in his book 
\cite[p.214]{Mat}, using his method of spines for Haken manifolds.) It was
pointed out to us by Henry Wilton, that there remains an orientation problem
when considering connected sums. If $M$ and $N$ denote two closed orientable 
$3$--manifolds, and $\overline{N}$ denotes $N$ with the opposite
orientation, it is not clear how to decide whether the connected sum $M\#%
\overline{N}$ is homeomorphic to $M\#N$ (though the geometric pieces are
clearly homeomorphic).

In \cite{M}, Manning gave an algorithm to decide if there exists a
hyperbolic structure on a closed orientable $3$--manifold given by a
triangulation. If there is such a structure, then Manning's algorithm
constructs a finite sided polyhedral fundamental region in hyperbolic $3$%
--space. In section 2 of this paper we use Manning's work to give a new
algorithm to decide whether or not two closed $3$--manifolds $M_{1}$ and $%
M_{2}$ are homeomorphic, when they are given by triangulations and known to
be hyperbolic. Manning's algorithm will construct a finite sided polyhedral
fundamental region $P_{i}$ in $\mathbb{H}^{3}$ for each manifold $M_{i}$,
and we show how to estimate how many copies of $P_{2}$ must be glued
together to contain $P_{1}$ and vice versa. This allows one to bound the
number of maps from the generators of $\pi _{1}(M_{1})$ to $\pi _{1}(M_{2})$
which might give rise to an isomorphism, and vice versa, enabling one to
check each such map in turn and decide whether or not it is an isomorphism.

In section 3 of this paper, we take a more algebraic point of view. The main
ingredients are the existence of a bi-automatic structure on the fundamental
groups of most geometric $3$--manifolds and the fact that Perelman's work
essentially reduces the homeomorphism problem for closed orientable
irreducible $3$--manifolds to the isomorphism problem for their fundamental
groups (with the notable exception of lens spaces).

Given a triangulated closed orientable 3-manifold, Jaco and Tollefson \cite[%
\S 7]{JT} give an algorithm producing a decomposition into irreducible
pieces. When the manifold is Haken, they give \cite[\S 8]{JT} algorithms for
finding its JSJ decomposition and finding the Seifert invariants for each
Seifert piece. Thus for a Haken closed Seifert fibre space, their algorithm
will yield all the Seifert invariants and hence a geometric structure. Now
suppose $M$ is a closed orientable $3$--manifold which is known to satisfy
Thurston's Geometrization Conjecture. In \cite[1.4]{M}, Manning combined the
algorithms in \cite{JT} with his algorithm for finding hyperbolic structures
to obtain an algorithm which finds the geometric structures on all the
pieces of $M$. In this section we describe a somewhat simpler algorithm,
with substantial algebraic ingredients, which can answer a somewhat simpler
question. Namely, if $M$ is a closed orientable $3$--manifold which is known
to be geometric, then our algorithm can decide on which geometry $M\ $is
modelled. To answer this simpler question, we do not need Manning's
algorithm for hyperbolic manifolds, and our procedure differs from that of
Manning in the other cases. One natural application of our algorithm would
be when $M$ is orientable and irreducible and not Haken, as then Perelman's
work implies that $M$ must be geometric. Our approach also works in the
Haken case and is rather different from that in \cite{JT}. Our algorithm
starts by finding a bi--automatic structure for the fundamental group of $M$%
, when such a structure exists. As for many algorithms for $3$--manifolds,
the main topological ingredient is an algorithm due to Jaco and Oertel \cite%
{JO} to decide whether a triangulated $3$--manifold is Haken, and if so, to
produce embedded incompressible surfaces.

\section{A new algorithm to decide if two hyperbolic $3$--manifolds are
homeomorphic}

Let $M_{1}$ and $M_{2}$ be closed $3$--manifolds each given by a
triangulation and known to be hyperbolic. It is straightforward to use these
triangulations to write down finite presentations for $\pi _{1}(M_{1})$ and $%
\pi _{1}(M_{2})$. In order to apply Manning's algorithm in \cite{M}, we
first need an algorithm for solving the word problem in each $\pi
_{1}(M_{i}) $. As $M_{i}$ is known to be hyperbolic, its fundamental group
is word hyperbolic and so bi-automatic. Then, the algorithm of \cite[3.4.1]%
{DBAE} (extended to the bi-automatic case, as stated in \cite{GS} after
Lemma 8.2) produces a bi--automatic structure. Alternatively Papasoglu \cite%
{P} gives an algorithm for calculating the hyperbolicity constant $\delta $
of the presentation, and with that a solution to the word problem is easily
built as in \cite{ABC}. It follows that there is an algorithm which can be
applied to the given presentation for $\pi _{1}(M_{i})$ which will find the
bi-automatic structure. (We note that chapter 5 of \cite{DBAE} gives the
algorithm to find automatic structures; the additional languages and axioms
needed to extend to bi-automatic structures are easily added.) Once this is
done, there is an algorithm for solving the word problem in $\pi _{1}(M_{i})$
(alternatively one can use the fact that the groups are residually finite).
Now, for each $M_{i}$, Manning's algorithm in \cite{M} constructs a convex
finite sided polyhedral fundamental region $P_{i}$ in $\mathbb{H}^{3}$. His
algorithm describes the vertices, edges and faces of $P_{i}$. In addition,
it describes the face pairing isometries needed to recover $M_{i}$ from $%
P_{i}$. It is now straightforward to write down a new presentation for $\pi
_{1}(M_{i})$ with the face pairing isometries of $P_{i}$ as generators. In
what follows we will use these presentations rather than the ones obtained
from the initial triangulations.

Now we consider the tiling of $\mathbb{H}^{3}$ by translates of $P_{2}$.
Given a union $X$ of such translates, we let $star(X)$ denote the union of
all translates of $P_{2}$ which meet $X$ in at least one point, and let $%
star^{n}(X)$ denote the result of applying the star operation $n$ times to $%
X $. Thus $star^{2}(X)=star(star(X))$. It may sometimes be convenient to
write $star^{0}(X)=X$. We consider the sequence $star^{n}(P_{2})$, $n\geq 0$%
, of subsets of $\mathbb{H}^{3}$. As the union of these subsets equals $%
\mathbb{H}^{3}$, their diameters must tend to infinity. We will refine this
obvious fact in the following way. We will show in Lemma \ref{algorithm}
below that there is an algorithm to find a positive number $R$, such that $%
star^{n}(P_{2})$ contains the metric ball $B(P_{2},nR)$.

Assuming this lemma for the moment, we now proceed as follows. If $M_{1}$
and $M_{2}$ are homeomorphic, Mostow's Rigidity Theorem implies that they
must be isometric with their hyperbolic metrics. This yields an isometry $%
\varphi $ from $\mathbb{H}^{3}$ tiled by translates of $P_{1}$ to $\mathbb{H}%
^{3}$ tiled by translates of $P_{2}$. By composing with the action of an
element of $\pi _{1}(M_{2})$ on $\mathbb{H}^{3}$, we can suppose that $%
\varphi (P_{1})$ meets $P_{2}$. Let $d_{1}$ denote the diameter of $P_{1}$.
Then $\varphi (P_{1})$ must be contained in the metric ball $B(P_{2},d_{1})$%
. If $n$ is an integer such that $nR>d_{1}$, it follows that $\varphi (P_{1})
$ is contained in $star^{n}(P_{2})$. Let $\alpha $ be a face pairing
isometry of $P_{1}$. Then the isometry $\varphi \alpha \varphi ^{-1}$ of $%
\mathbb{H}^{3}$ pairs faces of $\varphi (P_{1})$. As $\alpha $ lies in $\pi
_{1}(M_{1})$, the isometry $\varphi \alpha \varphi ^{-1}$ lies in $\pi
_{1}(M_{2})$, and so preserves the tiling by translates of $P_{2}$. Hence $%
\varphi \alpha \varphi ^{-1}$ must send a certain translate of $P_{2}$ which
is contained in $star^{n}(P_{2})$ to another such. Pick a path in $%
star^{n}(P_{2})$ which joins the interiors of these two translates of $P_{2}$%
, does not meet any edges and is transverse to the faces, and let $N$ denote
the number of times this path meets a face. Then $\varphi \alpha \varphi
^{-1}$ can be written as a word in the face pairing generators of $\pi
_{1}(M_{2})$ of length $N$.

Let $k_{1}$ denote the maximum number of translates of $P_{2}$ around an
edge, and $k_{2}$ denote the maximum number of translates of $P_{2}$ around
a vertex, and let $k$ denote the maximum of $k_{1}$ and $k_{2}$. Then in $%
star(P_{2})$, we can join{\ any point in the interior of} $P_{2}$ to{\ any
point in the interior of} any translate of $P_{2}$, by a path which crosses
at most $k$ faces. Note that a translate of $P_{2}$ in $star(P_{2})$ may
have just a vertex in common with $P_{2}$. It follows immediately that in $%
star^{n}(P_{2})$, we can join $P_{2}$ to any translate of $P_{2}$, by a path
which crosses at most $nk$ faces. Hence, in $star^{n}(P_{2})$, we can join
any two translates of $P_{2}$, by a path which crosses at most $2nk$ faces.
We conclude that if there is an isomorphism from $\pi _{1}(M_{1})$ to $\pi
_{1}(M_{2})$, there is one which maps each face pairing generator of $\pi
_{1}(M_{1})$ to a word of length no more than $2nk$ in the face pairing
generators of $\pi _{1}(M_{2})$. Similarly we can find integers $m$ and $%
k^{\prime }$ such that if there is an isomorphism from $\pi _{1}(M_{2})$ to $%
\pi _{1}(M_{1})$, there is one which is inverse to the previous one, which
maps each face pairing generator of $\pi _{1}(M_{2})$ to a word of length no
more than $2mk^{\prime }$ in the face pairing generators of $\pi _{1}(M_{1})$%
. This gives us a finite list of possible maps from generators of $\pi
_{1}(M_{1})$ to elements of $\pi _{1}(M_{2})$, and vice versa. For each such
map we can check whether it is a homomorphism, and for each pair of such
maps, can check if their composite is the identity. Each of these checks
again requires the solution of the word problem. Thus we can check whether
or not there is an isomorphism between $\pi _{1}(M_{1})$ and $\pi _{1}(M_{2})
$.

Before giving the proof of Lemma \ref{algorithm}, we need to discuss
algorithms for bounding various distances related to the convex polyhedra $%
P_{i}$.

We will consider the upper half space model of $\mathbb{H}^{3}$, and all
coordinates used will be euclidean. Recall that in this model a hyperbolic
geodesic is either a vertical line, or a semi-circle in a vertical plane
centred at some point of the base plane $\mathbb{R}^{2}$ of the model. 
The two points which form the intersection of such a semi-circle
with the base plane $\mathbb{R}^{2}$ of the model will be called the
boundary of the geodesic. If a hyperbolic geodesic is a vertical line in
this model, its boundary consists of one point in $\mathbb{R}^{2}$ and one
point at infinity. Also a hyperbolic plane in this model is either a
vertical plane or is a euclidean hemisphere centred at some point of the
base plane $\mathbb{R}^{2}$ of the model. The circle which is the
intersection of such a hemisphere with the base plane $\mathbb{R}^{2}$ of
the model will be called the boundary of the hyperbolic plane. If a
hyperbolic plane is a vertical plane in this model, its boundary consists of
a line in $\mathbb{R}^{2}$ and one point at infinity.

We recall that Manning's paper \cite{M} produces the hemispheres which
contain the faces of $P_{i}$, and that each of these hemispheres has centre
with coordinates which are algebraic numbers and has euclidean radius which
is also an algebraic number. If the intersection of two of these hemispheres
is non-empty, it is a semi-circle whose boundary points have coordinates
which are algebraic numbers and whose euclidean radius is also an algebraic
number. Further the vertices of $P_{i}$, each of which is the intersection
of three of these hemispheres, also have coordinates which are algebraic
numbers. These numbers can be approximated to any required degree of
accuracy over the rational numbers (to lie within an \textquotedblleft
isolating interval" with rational endpoints) using standard methods of
symbolic computation, as described by Manning (with reference to \cite{BW}
and \cite{L}). The algorithmic process starts from the results of Manning's
algorithm, which gives coordinates (i.e.{\ minimal polynomials and isolating
intervals for their roots}) for the vertices of the polyhedra $P_{i}$, and
for the centres of the semi-circles and hemispheres defining the edges and
faces of $P_{i}$, and estimates of their radii. Suppose that all these
isolating intervals are of width at most $\epsilon $, and to begin, suppose
that $\epsilon $ has been chosen so that $1/2^{c+1}<\epsilon <1/2^{c}$. We
refer to this number as the error in our calculations.

{In the upper half space model, after estimating euclidean distances, we can
then estimate hyperbolic distances using the usual formulae.} For simplicity
in the following seven statements, we will say that a point in the upper
half space model of $\mathbb{H}^{3}$ is algebraic if its coordinates are
algebraic, that an infinite geodesic in this model is algebraic if its
boundary points are algebraic (we count $\infty $ as being algebraic here),
that a hemisphere in this model is algebraic if its euclidean centre (in the
base plane $z=0$) and radius are algebraic, and that a vertical plane is
algebraic if it contains{\ at least two finite algebraic points (or has at
least one finite algebraic boundary point)}. 
Finally a compact geodesic segment is  algebraic if its
endpoints are algebraic.

\begin{enumerate}
\item The distance between two distinct algebraic points is not
algorithmically computable, but we can compute a positive lower bound which
will suffice for our requirements. The error incurred here is at most $%
2\epsilon $.

\item For this point and the next, we use the upper half space
model of the hyperbolic plane $\mathbb{H}^{2}$. Finding the distance
between an algebraic point $X$ and a disjoint algebraic geodesic $\lambda $
in the hyperbolic plane can be algorithmically reduced to finding the
distance between two algebraic points as follows. Let $\mu $ denote the
semi-circle through $X$ which meets $\lambda $ at right angles {and has its
centre on the base line $\mathbb{R}$. We can find this centre by solving
quadratic equations} specifying that the centre is equidistant from $X$ and
the point $\lambda \cap \mu $, and that the line from the centre to $\lambda
\cap \mu $ is tangent to $\lambda $. Hence we can also find $\lambda \cap
\mu $. Now the required distance equals the distance from $X$ to the point $%
\lambda \cap \mu $, so we can apply 1).

\item In the same way, finding the distance between two disjoint algebraic
geodesics (without a common boundary point) in the hyperbolic plane can be
algorithmically reduced to 1).

\item In the same way, finding the distance between an algebraic point $X$
and a disjoint algebraic{\ hyperbolic plane} $\Pi $ in hyperbolic $3$--space
can be algorithmically reduced to 1).

\item In the same way, finding the distance between two disjoint algebraic {%
hyperbolic planes} whose boundaries are also disjoint can be algorithmically
reduced to 1).

\item To find a lower bound for the distance between an algebraic{\ geodesic}
$\lambda $ and a disjoint algebraic{\ hyperbolic plane} $\Pi $, such that
the boundaries of $\lambda $ and $\Pi $ are also disjoint, we will choose an
algebraic{\ hyperbolic plane} $\Sigma $ which contains $\lambda $, and is
disjoint from $\Pi $, so that their boundaries are also disjoint. Once such $%
\Sigma $ has been found, the distance between $\Sigma $ and $\Pi $, which
can be found as in 5), gives the required lower bound.

If $\lambda $ is a vertical line, $\Pi $ must be a hemisphere, and we choose 
$\Sigma $ to be the vertical plane through $\lambda $ which is orthogonal to
the vertical plane which contains both $\lambda $ and the centre of the
hemisphere $\Sigma $.

Otherwise $\lambda $ is a semi-circle with endpoints $a$ and $b$ in the
plane $z=0$. Now we apply the Moebius transformation $\varphi $ given by $%
\varphi (z)=(z-a)/(z-b)$, which takes $\lambda $ to the vertical line $%
\lambda ^{\prime }$ above the origin, and takes $\Pi $ to an algebraic
hyperbolic plane $\Pi ^{\prime }$. 
The distance between $\lambda
^{\prime }$ and $\Pi ^{\prime }$ can be bounded below as in the preceding
two paragraphs. As $\varphi $ is a hyperbolic isometry, this is the required
lower bound for the distance between $\lambda $ and $\Pi $.

\item Finding the distance between an algebraic geodesic segment $e$ and a
disjoint algebraic{\ hyperbolic plane} $\Pi $ in hyperbolic $3$--space can
be algorithmically reduced to the preceding cases as follows. Let $\lambda $
denote the geodesic which contains $e$. If $\lambda $ and $\Pi $ are
disjoint and do not have a common boundary point, we can apply 6) to find a
lower bound for the distance between them. This is also a lower bound for
the distance between $e$ and $\Pi $. Otherwise, the distance between $e$ and 
$\Pi $ equals the distance between $\partial e$ and $\Pi $, which reduces
the problem to 4).
\end{enumerate}

\begin{lemma}
\label{algorithm}There is an algorithm to find a positive number $R$, such
that $star^{n}(P_{2})$ contains the metric ball $B(P_{2},nR)$.
\end{lemma}

\begin{proof}
We will find $R$ such that for each $n\geq 1$, the $R$--neighborhood of $%
\partial star^{n}(P_{2})$, does not meet $star^{n-1}(P_{2})$. Thus any path
in $\mathbb{H}^{3}$ which starts on $\partial star^{n}(P_{2})$ and ends on $%
\partial star^{n-1}(P_{2})$ must have length at least $R$. By induction it
follows that any path which starts on $\partial star^{n}(P_{2})$ and ends on 
$\partial P_{2}$ must have length at least $nR$. It follows immediately that 
$star^{n}(P_{2})$ contains the metric ball $B(P_{2},nR)$, as required.

We first give a description of the exact calculation before considering the
error term. We use the above seven points to find positive lower bounds for
various distances.

The distance between disjoint vertices of $P_{2}$ can be{\ bounded below
using 1)}.

The distance between a vertex $v$ of $P_{2}$ and a disjoint edge $e$ of $%
P_{2}$ can be estimated as follows. Let $\lambda $ denote the geodesic which
contains $e$. As $v$ cannot lie on $\lambda $, it suffices to estimate the
distance of $v$ from $\lambda $. If $\lambda $ is a vertical line, this can
be done as in 2). Otherwise, as in 6), let $a$ and $b$ denote the endpoints
of $\lambda $, and apply the Moebius transformation $\varphi $ given by $%
\varphi (z)=(z-a)/(z-b)$. This takes $\lambda $ to a vertical line $\lambda
^{\prime }$, and takes $v$ to an algebraic point $v^{\prime }$, so we can
now estimate the distance of $v^{\prime }$ from $\lambda ^{\prime }$ as in
2).

The distance between a vertex $v$ of $P_{2}$ and a disjoint face $F$ of $%
P_{2}$ can be estimated using 4), as $v$ cannot lie in the plane which
contains $F$.

The distance between disjoint edges $e$ and $f$ of $P_{2}$ can be estimated
as follows. Let $\lambda $ and $\mu $ denote the geodesics which contain $e$
and $f$ respectively. If $\lambda $ meets $\mu $ at a finite point or at
infinity, the distance between $e$ and $f$ is equal to the distance between $%
\partial e$ and $\partial f$, which can be bounded as in 1). If $\lambda $
and $\mu $ are disjoint, and disjoint at infinity, we will find an algebraic
plane $\Pi $ which contains $\lambda $ and is disjoint from $\mu $, and is
also disjoint from $\mu $ at infinity. If $\lambda $ is a vertical line, so
that $\mu $ must be a semi-circle, we take $\Pi $ to be the vertical plane
through $\lambda $ which is parallel to the base line of the semi-circle $%
\mu $. If $\lambda $ is not a vertical line, then, as in 6), we can apply a
Moebius transformation which takes $\lambda $ to a vertical line, and takes $%
\mu $ to an algebraic geodesic.

The distance between an edge $e$ of $P_{2}$ and a disjoint face $F$ of $%
P_{2} $ can be estimated using 7), as $e$ must be disjoint from the plane
which contains $F$.

{Finally the distance between disjoint faces $E$ and $F$ of $P_{2}$ can be
estimated as follows. Let $\Pi _{E}$ and $\Pi _{F}$ denote the planes which
contain $E$ and $F$ respectively. If these planes are disjoint, and disjoint
at infinity, we can estimate the distance between them using 5), and this
will be a lower bound for the distance between $E$ and $F$. Otherwise, the
distance between $E$ and $F$ is bounded below by the distance between $E$
and $\Pi _{F}$. This last distance equals the distance between $\Pi _{F}$
and some edge of $E$, and so equals one of the numbers already estimated. }

Now let $R$ denote half the minimum of all these numbers.

Let $W$ denote a vertex, edge or face of $P_{2}$. Then the definition of $R$
implies that the $R$--neighborhood of $W$ meets only those vertices, edges
or faces of $P_{2}$ which meet $W$.

Hence if $Q$ is a translate of $P_{2}$ in $star^{n}(P_{2})$ which meets $%
\partial star^{n}(P_{2})$, then the $R$--neighborhood of $\partial
star^{n}(P_{2})$ does not meet any vertex, edge or face of $Q$ except those
which meet $\partial star^{n}(P_{2})$. In particular, it follows that the $R$%
--neighborhood of $\partial star^{n}(P_{2})$, does not meet $%
star^{n-1}(P_{2})$, as required. Note that $\partial star^{n}(P_{2})$ and $%
star^{n-1}(P_{2})$ are disjoint.

In the actual algorithm, when dealing with approximations, all the
calculations above incur increasing error, but the fact that the number of
operations is finite means that there is a constant $C>0$ such that the
error in the estimate of each of these numbers is at most $C \epsilon$, so
that $R$ must be replaced by $R-C\epsilon$. It is of course possible that $%
R<C\epsilon$ in which case the algorithm must restart, replacing $\epsilon$
by $\epsilon/2$, recalculating the coordinates in Manning's algorithm to
this increased degree of accuracy, and recalculating $R$. Continue to do so
until $R>C\epsilon$, and then replace $R$ by $R-C\epsilon$, once this number
is positive.
\end{proof}

\section{Algorithm to find the geometry}

In this section we consider closed orientable irreducible geometric 
$3$--manifolds, given by finite triangulations. There are eight geometries{\
(as discussed in Scott's article \cite{Sc})}, and it is well known that a
closed $3$--manifold can have a geometric structure modelled on at most one
of these geometries. This can be proved by exhibiting properties of the
fundamental groups which distinguish the geometries. For example only closed
manifolds modelled on $S^{3}$ can have finite fundamental group. The point
of what we do in this section is that we can decide algorithmically on which
geometry a given geometric manifold is modelled. Such an algorithm is
described by Manning in \cite[1.4]{M}, but here we provide a more algebraic
treatment.

As we are considering orientable irreducible $3$--manifolds, the geometry $%
S^{2}\times \mathbb{R}$ cannot occur. For the only closed orientable
manifolds modelled on this geometry are $S^{2}\times S^{1}$ and $\mathbb{R}%
P^{3}\#\mathbb{R}P^{3}$, neither of which is irreducible. We start by
listing the remaining seven geometries together with some selected
properties of the closed manifolds modelled on these geometries.

\smallskip \renewcommand{\arraystretch}{1.2} \noindent\ 
\begin{tabular}{|c|p{10.1cm}|}
\hline
Geometry & Selected properties of any closed orientable $3$--manifold $M$
modelled on given geometry \\ \hline\hline
$S^{3}$ & $\pi _{1}(M)$ is finite. \\ \hline
$\mathbb{E}^{3}$ & $\pi _{1}(M)$ is virtually $\mathbb{Z}^{3}$, and $M$ is
Haken. Any two-sided incompressible surface in $M$ must be a torus. \\ \hline
$\mathbb{H}^{3}$ & $\pi _{1}(M)$ has no subgroup isomorphic to $\mathbb{Z}%
^{2}$. \\ \hline
$\mathbb{H}^{2}\times \mathbb{R}$ & $M$ is a Seifert fibre space with
hyperbolic base orbifold. $M$ is Haken, and contains an embedded
incompressible hyperbolic surface. \\ \hline
$Nil$ & $M$ is a Seifert fibre space with Euclidean base orbifold. \\ \hline
$Solv$ & $M$ is Haken. Any two-sided incompressible surface in $M$ must be a
torus. \\ \hline
$\widetilde{SL_{2}\mathbb{R}}$ & $M$ is a Seifert fibre space with
hyperbolic base orbifold. A two-sided incompressible surface in $M$ must be
a torus. \\ \hline
\end{tabular}

\smallskip

A crucial fact for us is that if $M$ is modelled on one of the above seven
geometries, then $\pi _{1}(M)$ is bi-automatic, except in the cases when the
geometry is $Nil$ or $Solv$ (\cite[chapter 12]{DBAE} proves automaticity).
In order to apply the theory of bi-automatic structures, we first need to be
able to decide whether $M$ is modelled on $Nil$ or $Solv$.

The key topological algorithm we will need is that of Jaco and Oertel \cite%
{JO} which decides whether a given triangulated $3$--manifold $M$ is Haken.
Their paper also shows how to decide whether $M$ has an incompressible
surface which is a torus, and how to find such a torus. We will use these
algorithms several times in what follows.

As usual $T$ denotes the $2$--torus $S^{1}\times S^{1}$. We also need to be
able to decide whether{\ a compact orientable manifold $M^{\prime }$ is
homeomorphic to $T\times I$. If $M^{\prime }$ is irreducible,} this is a
special case of algorithm 9.7 of \cite{JT}. Note that the algorithms of Jaco
and Oertel also find essential discs, and in this case cutting along a
properly embedded disc gives a $3$--ball, which can be recognised (by
Rubinstein \cite{R} and Thompson \cite{T}, or by Perelman's solution of the
Poincar\'{e} conjecture).

\begin{lemma}
(\textit{cf.} Theorem 5.5 of \cite{Sc}) If $M$ is orientable and is obtained
from $T\times I$ by gluing $T\times \{0\}$ to $T\times \{1\}$ by some
homeomorphism $h$, then $M$ is geometric and is modelled on one of $\mathbb{E%
}^{3}$, $Nil$ or $Solv$.
\end{lemma}

\begin{proof}
The action of $h$ on $H_{1}(T)\cong \mathbb{Z}^{2}$ is given by an integer $%
2\times 2$ matrix $A$. We consider the trace, $tr(A)$, of $A$. If $%
\left\vert tr(A)\right\vert <2$, or if $A=\pm I$, then $A$, and hence $h$,
must be periodic, so that $M$ is modelled on $\mathbb{E}^{3}$. If $%
\left\vert tr(A)\right\vert >2$, then $A$ has distinct real eigenvalues, so
that $M$ is modelled on $Solv$. If $\left\vert tr(A)\right\vert =2$, then $A$
has a repeated eigenvalue equal to $\pm 1$. So long as $A\neq \pm I$, this
implies that $M$ is modelled on $Nil$.
\end{proof}

\begin{remark}
Suppose that we have found an incompressible torus $T$ in an orientable $3$%
--manifold $M$, using the algorithms of normal surface theory, and that we
have checked that cutting $M\ $along $T$ yields a manifold homeomorphic to $%
T\times I$. In this situation, one can algorithmically calculate the action
of $h$ on $H_{1}(T)\cong \mathbb{Z}^{2}$, and so can decide on which
geometry $M$ is modelled.
\end{remark}

We now discuss the geometries $Nil$ and $Solv$ in more detail, and describe
an algorithm to decide which geometry occurs.

If $M$ is modelled on $Solv$, then $M$ is a bundle over a $1$--dimensional
orbifold with fibre the torus. Thus either $M$ is a bundle over $S^{1}$ with
fibre the torus, or $M$ is double covered by such a manifold.

If $M$ is modelled on $Nil$, there are several cases. If $M$ is Haken, then
it is a Seifert fibre space whose base orbifold is a torus, Klein bottle, $%
S^{2}(2,2,2,2)$ or $P^{2}(2,2)$. If this orbifold is not a torus, there is a
regular cover of $M$ of degree $2$ or $4$ whose base orbifold is a torus. If 
$M$ is not Haken, it is a Seifert fibre space whose base orbifold is $%
S^{2}(p,q,r)$, where $(p,q,r)$ is one of $(3,3,3)$, $(2,2,4)$ or $(2,3,6)$.
Now the orbifold fundamental group of $S^{2}(p,q,r)$ is the triangle group $%
\Delta (p,q,r)$, and in these cases, $\Delta (p,q,r)$ has a homomorphism to $%
\mathbb{Z}_{3}$, $\mathbb{Z}_{4}$ or $\mathbb{Z}_{6}$ with kernel isomorphic
to $\mathbb{Z}^{2}$. Thus there is a homomorphism of $\pi _{1}(M)$ to $\{1\}$%
, $\mathbb{Z}_{3}$, $\mathbb{Z}_{4}$ or $\mathbb{Z}_{6}$ whose kernel
determines a finite cover of $M$ which is a circle bundle over the torus. If 
$M$ is modelled on $Nil$ and is a circle bundle over the torus, then any
two-sided incompressible surface in $M$ must be a vertical torus, and if we
cut $M$ along such a torus, the result will be homeomorphic to $T\times I$.

We can apply the preceding paragraph as follows. For any triangulated closed
orientable irreducible $3$--manifold $M$, we can check whether there is a
homomorphism of $\pi _{1}(M)$ to $\{1\}$, $\mathbb{Z}_{2}$, $\mathbb{Z}_{3}$%
, $\mathbb{Z}_{4}$, $\mathbb{Z}_{2}\times \mathbb{Z}_{2}$, or $\mathbb{Z}%
_{6} $ whose kernel determines a finite cover of $M$ with infinite first
homology group. If this does not occur, then $M$ cannot be modelled on $Nil$
or $Solv$. If this does occur, we can check whether the covering contains an
incompressible torus, and if it does, we can check whether cutting along
this torus yields $T\times I$. If this occurs, the remark above tells us how
to determine the geometry on this finite cover and hence on $M$.

Thus we can decide whether or not $M$ is modelled on $Nil$ or $Solv$, and if
it is so modelled can decide which.

This reduces us to considering the five remaining geometries.

Given a triangulated closed orientable irreducible $3$--manifold $M$, we can
write down a presentation for $\pi _{1}(M)$. If we know that $\pi _{1}(M)$
is bi-automatic, we can algorithmically find a bi-automatic structure. {Part
of this structure is a regular (or rational) language of representatives for
the elements of the group. We can suppose, using Theorem 2.5.2 of \cite{DBAE}%
, that each group element has a unique representative in the language, and
it is easy to check whether a regular language is finite or infinite (see
for instance \cite{HU} Theorem 3.7). Thus we can algorithmically check
whether $\pi _{1}(M)$ is finite, and so can decide whether $M$ is modelled
on $S^{3}$.}

This reduces us to the four remaining geometries, which are $\mathbb{E}^{3}$%
, $\mathbb{H}^{3}$, $\mathbb{H}^{2}\times \mathbb{R}$ and $\widetilde{SL_{2}%
\mathbb{R}}$.

If $M$ is modelled on $\mathbb{E}^{3}$, there are several cases. In all
cases, $M$ is Haken. It is a Seifert fibre space whose base orbifold is a
torus, Klein bottle, $S^{2}(2,2,2,2)$, $P^{2}(2,2)$, or $S^{2}(p,q,r)$,
where $(p,q,r)$ is one of $(3,3,3)$, $(2,2,4)$ or $(2,3,6)$. Thus, as for $%
Nil$ geometry, there is a homomorphism of $\pi _{1}(M)$ to $\{1\}$, $\mathbb{%
Z}_{2}$, $\mathbb{Z}_{3}$, $\mathbb{Z}_{4}$, $\mathbb{Z}_{2}\times \mathbb{Z}%
_{2}$, or $\mathbb{Z}_{6}$ whose kernel determines a finite cover of $M$
whose base orbifold is the torus. But now this finite cover must be a $3$%
--torus, and so have free abelian fundamental group.

Thus to decide whether or not $M$ is modelled on $\mathbb{E}^{3}$, we simply
check whether there is a homomorphism of $\pi _{1}(M)$ to $\mathbb{Z}_{2}$, $%
\mathbb{Z}_{3}$, $\mathbb{Z}_{4}$, $\mathbb{Z}_{2}\times \mathbb{Z}_{2}$, or 
$\mathbb{Z}_{6}$ whose kernel is free abelian of rank $3$. For $M\ $is
modelled on $\mathbb{E}^{3}$, if and only if there is such a homomorphism.

This reduces us to the three remaining geometries, which are $\mathbb{H}^{3}$%
, $\mathbb{H}^{2}\times \mathbb{R}$ and $\widetilde{SL_{2}\mathbb{R}}$.

Using the bi-automatic structure on $\pi _{1}(M)$ one can check the answers
to the following questions. Does $\pi _{1}(M)$ have nontrivial centre? Does $%
\pi _{1}(M)$ have a subgroup of index $2$ with nontrivial centre? {(A
regular language for the centre of a bi-automatic group is constructed in
Cor. 4.4.1 of \cite{GS1}, and as noted earlier, it is easy to check whether
the language is infinite or finite, and in the latter case deduce how many
elements are in the centre. It is also straightforward to obtain a
bi-automatic structure for all subgroups of a given finite index, as in
Theorem 4.1.1 of \cite{DBAE}.)} If the answer to both questions is negative,
then $M$ must be hyperbolic. If we find a positive answer, then $M$ must be
modelled on one of $\mathbb{H}^{2}\times \mathbb{R}$ and $\widetilde{SL_{2}%
\mathbb{R}}$. To distinguish these cases, we use the facts that if $M$ is
modelled on $\widetilde{SL_{2}\mathbb{R}}$, then any incompressible surface
in $M$ must be a (vertical) torus, whereas if $M$ is modelled on $\mathbb{H}%
^{2}\times \mathbb{R}$, there must be horizontal incompressible surfaces in $%
M$ none of which can be a torus. Thus we can apply the algorithm of Jaco and
Oertel \cite{JO} to decide whether $M$ contains an incompressible surface
which is not a torus.

The referee pointed out an alternative algebraic approach to distinguishing
the $\mathbb{H}^{2}\times \mathbb{R}$ and $\widetilde{SL_{2}\mathbb{R}}$
cases. It is based on two observations. The first is that if $M$ is modelled
on one of these two geometries, then $M$ has a finite cover $M_{1}$ which is
a bundle over a surface with fibre the circle, such that the centre of $\pi
_{1}(M_{1})$ is infinite cyclic. The second is that if $M_{1}$ is such a
manifold, one can decide on which geometry $M_{1}$ (and hence $M$) is
modelled by checking whether the centre of $\pi _{1}(M_{1})$ injects into $%
H_{1}(M_{1})$. If it does, then the geometry is $\mathbb{H}^{2}\times 
\mathbb{R}$, and if it does not, then the geometry is $\widetilde{SL_{2}%
\mathbb{R}}$. It will be helpful to add a third observation. This is that if
the centre of $\pi _{1}(M)$ is infinite cyclic, and if {the centre of $\pi
_{1}(M)$ injects into $H_{1}(M)$, then }$M$ must be modelled on {$\mathbb{H}%
^{2}\times \mathbb{R}$. This is true because the assumption that the centre
of $\pi _{1}(M)$ injects into $H_{1}(M)$ immediately implies that the centre
of $\pi _{1}(M_{1})$ injects into $H_{1}(M_{1})$.}

To distinguish these two geometries algorithmically, we proceed as follows.
Suppose that $\pi _{1}(M)$ has nontrivial centre. (If not, replace $M$ by
the double cover which does have this property.) Now this centre, which we
denote by $A$, is infinite cyclic and the quotient {$\pi _{1}(M)/A$ is a
Fuchsian group $\Gamma ${. If $\Gamma $ is torsion free, it is a surface
group, and we take $M_{1}$ to equal $M$. We can decide whether $\Gamma $ is
torsion free. First $\Gamma $ is $\delta $--hyperbolic for some $\delta >0 $%
, and such a $\delta $ can be found algorithmically (see for instance \cite%
{P}). Now any torsion element of $\Gamma $ must have length at most $4\delta
+2$ (as in the proof of Theorem III.$\Gamma .3.2$ of \cite{Bridson-Haefliger}%
) and so the orders of all torsion elements can be found. If there is any
nontrivial torsion, we let $k$ denote the least common multiple of these
orders. It follows from Theorem 1 of \cite{EEK} that there is a torsion free
subgroup $\Gamma _{1}$ of index $2k$ in $\Gamma $. Thus there is a degree $%
2k $ cover $M_{1}$ of $M$, such that $\pi _{1}(M_{1})$ has centre $A$, and $%
\pi _{1}(M_{1})/A$ is the surface group $\Gamma _{1} $. {We can now
algorithmically find finite presentations for the (finitely many) subgroups
of $\Gamma $ of index $2k$, one of which is $\Gamma _{1}$. At this point one
could simply check each of these subgroups of $\Gamma $ to decide whether it
is torsion free, but it seems simpler to proceed in the following way.
Instead consider the index $2k$ subgroups of $\pi _{1}(M)$ with centre $A$
and quotient one of the index $2k$ subgroups of $\Gamma $. For each such
subgroup $G$ of $\pi _{1}(M)$, we check whether $A$ injects into the
abelianisation of $G$. If this does not occur for any of these subgroups we
immediately deduce that $M$ has $\widetilde{SL_{2}\mathbb{R}}$ geometry. If
this does occur for some such subgroup, say $\pi _{1}(M_{2})$, we can apply
the third observation above to see that $M_{2}$, and hence $M$, has $\mathbb{%
H}^{2}\times \mathbb{R}$ geometry, without ever needing to check whether the
quotient $\pi _{1}(M_{2})/A$ is torsion free. }}}


\begin{thebibliography}{99}
\bibitem{ABC} J. Alonso, T. Brady, D. Cooper, V. Ferlini, M. Lustig, M.
Mihalik, M. Shapiro, and H. Short, \textit{Notes on word hyperbolic groups},
Edited by Short. \textit{Group theory from a geometrical viewpoint}
(Trieste, 1990), 3--63, World Sci. Publ., River Edge, NJ, 1991.

\bibitem{BW} T. Becker and V. Weispfenning, \textit{Gr\"{o}bner bases. A
computational approach to commutative algebra}. In cooperation with Heinz
Kredel. Graduate Texts in Mathematics, 141. Springer-Verlag, New York, 1993.

\bibitem{BBMBP} L. Bessi\`{e}res, G. Besson, S. Maillot, M. Boileau, and J.
Porti, \textit{Geometrisation of }$3$\textit{--manifolds}, EMS Tracts in
Mathematics, 13. European Mathematical Society (EMS), Z\"{u}rich, 2010.

\bibitem{Bridson-Haefliger} M. R. Bridson and A. Haefliger, \textit{Metric
spaces of non-positive curvature}, Grundlehren der Mathematischen
Wissenschaften [Fundamental Principles of Mathematical Sciences], 319.
Springer-Verlag, Berlin, 1999.

\bibitem{DBAE} D.~B.~A. Epstein, with J.~W. Cannon, D.~F. Holt,
S.~V.~F.~Levy, M.~S. Paterson and W.~P. Thurston, \textit{Word processing in
Groups}, Jones and Bartlett, Boston, 1992.{\ }

\bibitem{EEK} A.~L. Edmonds, J.~H. Ewing and R.~S. Kulkarni, \textit{Torsion
free subgroups of Fuchsian groups and tessellations of surfaces\/}, Bull.
A.M.S., 6 (1982), 456--458.

\bibitem{GS1} S.M. Gersten and H. Short, \textit{Rational subgroups of
bi-automatic groups}, Ann of Math, 134 (1991), 125--158.

\bibitem{GS} S.M. Gersten and H. Short, \textit{Small cancellation theory
and automatic groups: Part II}, Invent Math \textbf{105}, (1991), 641--662.

\bibitem{H} G. Hemion, \textit{On the classification of homeomorphisms of }$%
2 $\textit{--manifolds and the classification of }$3$\textit{--manifolds},
Acta Math. 142 (1979), no. 1-2, 123--155.

\bibitem{J} W.H. Jaco, Peking Summer School Lectures 2005, slides online at
http://www.math.okstate.edu/\symbol{126}Jaco/pekinglectures.htm

\bibitem{JO} W. Jaco and U. Oertel, \textit{An algorithm to decide if a }$%
\mathit{3}$\textit{--manifold is a Haken manifold}, Topology, \textbf{23},%
\textbf{\ }(1984), 195--209.

\bibitem{JT} W. Jaco and J.L. Tollefson, \textit{Algorithms for the complete
decomposition of a closed }$\mathit{3}$\textit{--manifold}, Illinois J Math, 
\textbf{39}, (1995), 358--406.

\bibitem{L} R. Loos, \textit{Computing in algebraic extensions}, Computer
algebra, 173--187, Springer, Vienna, 1983.

\bibitem{M} J. Manning, \textit{Algorithmic detection and description of
hyperbolic structures on closed }$\mathit{3}$\textit{--manifolds with
solvable word problem\/}, Geometry and Topology, vol 6, (2002), 1--25.

\bibitem{Mat} S. Matveev, \textit{Algorithmic Topoloy and Classification of
3--manifolds\/}, Algorithms and Computation in Mathematics 9, Springer,
Berlin, 2007.

\bibitem{HU} J. E. Hopcroft and J. D. Ullman, \textit{Introduction to
Automata Theory, Languages and computation}, Addison-Wesley, Reading, Mass.,
1979.

\bibitem{P} P. Papasoglu, \textit{An algorithm detecting hyperbolicity. }%
Geometric and computational perspectives on infinite groups (Minneapolis, MN
and New Brunswick, NJ, 1994), 193--200, DIMACS Ser. Discrete Math. Theoret.
Comput. Sci., 25, Amer. Math. Soc., Providence, RI, 1996.

\bibitem{Perelman1} G. Perelman, \textit{The entropy formula for the Ricci
flow and its geometric applications}, preprint arxiv.org/abs/math/0211159.

\bibitem{Perelman2} G. Perelman, \textit{Ricci flow with surgery on
three-manifolds}, preprint arxiv.org/abs/math/0303109.

\bibitem{Perelman3} G. Perelman, \textit{Finite extinction time for the
solutions to the Ricci flow on certain three-manifolds}, preprint,
arxiv.org/abs/math/0307245.

\bibitem{R} J. H. Rubinstein, \textit{An algorithm to recognize the 3-sphere}%
, Proceedings of the International Congress of Mathematicians, Vol. 1, 2 (Z%
\"{u}rich, 1994), 601--611, Birkh\"{a}user, Basel, 1995.

\bibitem{Sc} G.P. Scott, \textit{The geometries of 3--manifolds}, Bull.
London Math. Soc. 15 (1983), no. 5, 401--487.

\bibitem{S} Z. Sela, \textit{The isomorphism problem for hyperbolic groups. I%
}, Ann of Math (2), \textbf{141} (1995), 217--283.

\bibitem{T} A. Thompson, \textit{Thin position and the recognition problem
for }$S^{3}$, Math. Res. Lett. 1 (1994), no. 5, 613--630.

\bibitem{wpt} W. P. Thurston, \textit{Three Dimensional Manifolds, Kleinian
Groups and Hyperbolic Geometry\/}, Bull. A.M.S. 6 (1982), 357--381.
\end{thebibliography}
\end{document}